\documentclass[10pt,onecolumn]{IEEEtran}

\usepackage{amssymb,amsmath, amsfonts, theorem}
\usepackage{graphicx}
\usepackage{epsfig}
\usepackage{algorithm,algorithmic,xspace}
\usepackage{times}
\usepackage{color}
\usepackage{flafter}
\usepackage{psfrag}
\usepackage{pstricks}
\usepackage[tight,footnotesize]{subfigure}
\usepackage{bbm}
\usepackage{float}
\usepackage{pgfplots}
\pgfplotsset{compat=1.3}
\usepackage[Symbol]{upgreek}
\usepackage{tikz}
\usetikzlibrary{matrix}
\usetikzlibrary{positioning}
\usetikzlibrary{arrows,shapes}
\usepackage{pst-pdf}
\usepackage{subfigure}
\usepackage{cases}
\usepackage{enumerate}
\usepackage{authblk}
\usepackage{fullpage}
\usepackage{multicol}
\usepackage{cite}

 \newcommand{\mc}{\mathcal}

\newcommand{\reals}{{\mathbb R}}

\newcommand{\ones}{\mathbbm{1}}
\newcommand{\ms}{\scriptscriptstyle}
\newcommand{\Etree}{E_{{\ms \mc{T}}}}
\newcommand{\Eforrest}{E_{{\ms \mc{F}}}}
\newcommand{\Ecycle}{E_{{\ms \mc{C}}}}
\newcommand{\Etreep}{E_{{\ms \mc{T}_+}}}

\newcommand{\Tcycle}{T_{{\ms (\mc{T},\mc{C})}}}
\newcommand{\Tcyclef}{T_{{\ms (\mc{F},\mc{C})}}}

\newcommand{\R}{R_{\ms (\mc{T},\mc{C})}}
\newcommand{\Rf}{R_{\ms (\mc{F},\mc{C})}}

\newcommand{\Rtplus}{R_{\ms (\mc{T}_{+},\mc{C}_{+})}}

\newcommand{\Edgelapf}{L_e(\mc{F})R_{\ms (\mc{F},\mc{C})}WR_{\ms (\mc{F},\mc{C})}^T}

 \newtheorem{theorem}{Theorem}[section]
 \newtheorem{proposition}[theorem]{Proposition}
 \newtheorem{corollary}[theorem]{Corollary}
 \newtheorem{definition}[theorem]{Definition}

 \def\QEDclosed{\mbox{\rule[0pt]{1.3ex}{1.3ex}}} 

\newcommand{\bea}{\begin{eqnarray}}
\newcommand{\eea}{\end{eqnarray}}
\newcommand{\beas}{\begin{eqnarray*}}
\newcommand{\eeas}{\end{eqnarray*}}
\newcommand{\leftm}{\left[\begin{array}}
\newcommand{\rightm}{\end{array}\right]}

 \def\QEDclosed{\mbox{\rule[0pt]{1.3ex}{1.3ex}}} 
 
\newcommand{\margin}[1]{\marginpar{\tiny\color{red} #1}}

\newcommand{\MBremove}[1]{\margin{removed by MB}}
\newcommand\oprocendsymbol{\hbox{$\square$}}
\newcommand\oprocend{\relax\ifmmode\else\unskip\hfill\fi\oprocendsymbol}

\title{\LARGE \bf On the Definiteness of the Weighted Laplacian\\ and its Connection to Effective Resistance}

\author{Daniel Zelazo$^1$ and Mathias B\"urger$^2$
\thanks{$^1$Daniel Zelazo is with the Faculty of Aerospace Engineering, Technion, Israel.  
        {\tt\small dzelazo@technion.ac.il}.}
\thanks{$^2$Mathias B\"urger is with the Cognitive Systems Group at Robert Bosch GmbH
        {\tt\small mathias.buerger@de.bosch.com}.}%
}

\begin{document}
\maketitle
\thispagestyle{empty}
\pagestyle{empty}

\begin{abstract}
This work explores the definiteness of the weighted graph Laplacian matrix with negative edge weights.  The definiteness of the weighted Laplacian is studied in terms of certain matrices that are related via congruent and similarity transformations.  For a graph with a single negative weight edge, we show that the weighted Laplacian becomes indefinite if the magnitude of the negative weight is less than the inverse of the effective resistance between the two incident nodes.  This result is extended to multiple negative weight edges.  The utility of these results are demonstrated in a weighted consensus network where appropriately placed negative weight edges can induce a clustering behavior for the protocol. 
\end{abstract}

\vspace{-10pt}
\section{Introduction}

The combinatorial graph Laplacian matrix is one of the most important and useful matrix representations of a graph.  The spectral properties of the graph Laplacian matrix can be used to study many combinatorial properties of a graph.  Well-known results in this venue include the \emph{Matrix-Tree Theorem} which states that the number of spanning trees in a graph is equal to any cofactor of the Laplacian matrix, or the \emph{algebraic connectivity} that relates the connectedness of a graph to the smallest non-zero eigenvalue of the Laplacian \cite{Fiedler1973, Godsil2001, Mohar1997}.  The graph Laplacian has also proved useful in the study of random walks and Markov chains, graph partitioning, spectral clustering, and more \cite{Gutman1994, Spielman2010,Luxburg2007}.  Within the controls community, the Laplacian matrix has taken a central role in the control and coordination of multi-agent systems due to its distributed structure and utility for problems related to formation control and synchronization \cite{Mesbahi2010}.

The notion of edge weights in a graph is a natural mathematical extension to the combinatorial theory of graphs.  Edge weights are also motivated by the modeling of physical processes \cite{Newman2003}, or as a design parameter in engineered systems used to improve certain performance metrics \cite{Xiao2004}.  For many reasons, edge weights are often taken to be non-negative numbers.  Indeed, in this case the weighted graph Laplacian matrix admits many favorable properties.  For example, the weighted Laplacian with positive weights belongs to the class of $Z-$matrices for which many results are known \cite{Horn1991}.  Another important property is that for undirected graphs with non-negative edge weights, the weighted Laplacian matrix is always positive semi-definite.

Recently there has been a growing interest in graphs containing negative edge weights.  In \cite{Xiao2004} it was shown that negative edge weights appear as an optimal solution for finding the fastest converging linear iteration used in distributed averaging.   The introduction of negative edge weights in problems related to the control of multi-agent systems can lead to steady-state configurations that are \emph{clustering} \cite{Qin2013, Xia2011a}.  In \cite{Altafini2013b}, negative weights are used to model antagonistic interactions in a social network and conditions are provided for when such weights lead to \emph{bipartite consensus}.  Finally, bounds on the number of positive, negative, and zero eigenvalues of the weighted Laplacian with negative weights are provided in \cite{Bronski2014}.

The study of the weighted Laplacian with negative edge weights is therefore of interest to a broad range of communities, and motivates the contributions of this work.  In particular, we examine conditions on how both the magnitude and location of negative weight edges in a weighted graph impact the definiteness of the weighted Laplacian.   This is achieved by first providing general results on how the signature of the weighted Laplacian is related to certain associated matrices, including the weighted edge Laplacian matrix \cite{Zelazo2009b} and another matrix related to the cut space of a graph \cite{Godsil2001}.  These results are then used to make conclusions on the definiteness of the weighted Laplacian.  
For the case of a graph with a single negative edge weight, we demonstrate that the definiteness of the Laplacian depends on the magnitude of that weight and is intimately related to the effective resistance between the incident nodes.  This result is also extended to graphs with multiple negative edge weights.  We demonstrate the utility of these results in the context of a linear weighted consensus protocol showing how careful selection of negative weight values can lead to a clustering steady-state configuration.

The organization of the paper is as follows.  Some basic mathematical preliminaries related to graph theory are given in the next sub-section.  In Section \ref{sec:signature}, results on the signature of the weighted Laplacian are provided.  The main results on the definiteness of the weighted Laplacian and the connection to effective resistance is presented in Section \ref{sec:definiteness}.  Section \ref{sec:example} shows how these results can be used in a weighted linear consensus protocol.  Finally, some concluding remarks are offered in Section \ref{sec:conclusion}.
%
\paragraph*{Preliminaries}
This work makes use of basic notions from algebraic graph theory \cite{Godsil2001}.  An undirected \emph{weighted graph} $\mc{G} = (\mc{V},\mc{E},\mc{W})$ is a triple consisting of the node set $\mc{V}$, edge set $\mc{E} \subseteq \mc{V} \times \mc{V}$, and weight function that maps each edge to a scalar value, $\mc{W}: \mc{E} \rightarrow \reals \setminus \{0\}$.
Note that we do not require the weights to be positive.  We often collect the weights of all the edges in a diagonal matrix $W \in \reals^{|\mc{E} | \times |\mc{E} |}$ such that $[W]_{kk} = \mc{W}(k)=w_k$ with $k = (i,j) \in \mc{E}$.  

A \emph{spanning tree} subgraph of $\mc{G}$ is a connected graph $\mc{T} = (\mc{V},\mc{E}_{\ms \mc{T}}) \subseteq \mc{G}$ that contains no cycles.  Similarly, a \emph{spanning forrest} subgraph of $\mc{G}$ is the graph $\mc{F} = (\mc{V},\mc{E}_{\ms \mc{F}}) \subseteq \mc{G}$ that contains no cycles (note that $\mc{F}$ can be a disconnected graph).
Every graph $\mc{G}$ can always be expressed as the union of a spanning tree (or forrest if the graph is not connected) and another subgraph containing the remaining edges, i.e., $\mc{G} = \mc{T} \cup \mc{C}$ ($\mc{G} = \mc{F} \cup \mc{C}$).  The subgraph $\mc{C}$ necessarily ``completes cycles" in $\mc{G}$, and is defined as $\mc{C} = (\mc{V},\mc{E}_{\ms \mc{C}}) \subset \mc{G}$ with $\mc{E}_{\ms \mc{C}} = \mc{E} \setminus \mc{E}_{\ms \mc{T}}$ (similarly defined with a forrest instead of tree); we refer to this as the \emph{cycle subgraph}.  

The \emph{incidence matrix} of a graph, $E(\mc{G}) \in \reals^{|\mc{V}| \times |\mc{E}|}$ is defined in the normal way.
With an appropriate labeling of the edges, we can always express the incidence matrix as {\small $E(\mc{G}) = \leftm{cc} E(\mc{T}) & E(\mc{C}) \rightm$} ({\small $E(\mc{G}) = \leftm{cc} E(\mc{F}) & E(\mc{C}) \rightm$}).  An important property of the incidence matrix is that $E(\mc{G})^T\ones = 0$ for any graph $\mc{G}$, where $\ones$ is the vector of all ones.  For a more compact notation, we will write $E := E(\mc{G}), \Etree = E(\mc{T}), \Eforrest = E(\mc{F})$, and $\Ecycle = E(\mc{C})$.  

\section{The Signature of the Weighted Laplacian}\label{sec:signature}

In this section we explore properties related to the \emph{signature} of the weighted Laplacian.\footnote{The {signature} of a real symmetric matrix $A$, denoted by the triple $\sigma(A)=(n_+,n_-,n_0)$, is the number of positive, negative, and zero eigenvalues of the matrix.}  Knowledge of the signature can be used, for example, to draw conclusions about the definiteness of that matrix.
%
An important result on the the signature of a matrix is \emph{Sylvester's Law of Inertia}, which states that all congruent symmetric matrices have the same signature \cite{Horn1985}.\footnote{A square matrix $A$ is \emph{congruent} to a square matrix $B$ of the same dimension if there exists an invertible matrix $S$ such that $B=S^TAS$.}  

Recall that for a weighted graph with only positive edge weights, one has $\sigma(L(\mc{G})) = (|\mc{V}|-c,0,c)$, where $c$ is the number of connected components of $\mc{G}$~\cite{Godsil2001}.  For a graph with negative weights, however, this is not true in general.  
%
Furthermore, the number of eigenvalues at the origin will no longer be a function of only the number of connected components in the graph.  


To understand how the presence of negative edge weights influences the definiteness of the weighted Laplacian, we consider the definiteness of certain associated matrices that are related via congruent transformations. In this direction, we first review the notion of the \emph{edge Laplacian} \cite{Zelazo2009b}, and provide here an extension for weighted graphs.
%
For a weighted graph $\mc{G}=(\mc{V},\mc{E},\mc{W})$, the \emph{weighted edge Laplacian} matrix is defined as
\bea\label{weighted_edge_lap}
L_e(\mc{G}) = W^{\frac{1}{2}}E^TEW^{\frac{1}{2}} \in \reals^{|\mc{E}| \times |\mc{E}|}.
\eea

We now present some basic results relating the weighted edge Laplacian matrix to the graph Laplacian.
\begin{proposition}\label{prop:edgelap_sim2}
The weighted Laplacian matrix $L(\mc{G})=EWE^T$ is similar to the matrix
\bea\label{essential_weighted_edgelap}
\leftm{cc} L_e(\mc{F}) \Rf W\Rf^T & {\bf 0} \\ {\bf 0} & {\bf 0}_c \rightm,
\eea
where $\mc{G}$ has $c$ connected components, $\mc{F} \subseteq \mc{G}$ is a spanning forrest of $\mc{G}$, and 
$$\Rf = \leftm{cc} I & L_e(\mc{F})^{-1}\Eforrest^T\Ecycle \rightm=  \leftm{cc} I & \Tcyclef \rightm.$$
\end{proposition}
\begin{proof}
Define the transformation matrices 
\beas
S = \leftm{cc} \Eforrest& N_{\ms {\mc{F}}} \rightm &,&
S^{-1} = \leftm{c} L_e(\mc{F})^{-1}\Eforrest^T \\ N_{\ms {\mc{F}}}^T \rightm,
\eeas
where $\mbox{\bf IM} [N_{\ms \mc{F}} ]= \mbox{\bf span}[\mc{N}(\Eforrest^T)]$.
It is straightforward to verify that the matrix in (\ref{essential_weighted_edgelap}) equals $S^{-1}L(\mc{G})S$.
\QEDclosed
\end{proof}

The matrix $L_e(\mc{F}) \Rf W\Rf^T:=L_{\ms ess}(\mc{F})$ is referred to as the \emph{essential edge Laplacian} \cite{Zelazo2011} (for a connected graph with spanning tree $\mc{T}$, we write $L_{\ms ess}(\mc{T})$).  Indeed, if $\mc{G}$ is connected, then $L_{\ms ess}(\mc{F})$ has the same non-zero eigenvalues as the weighted Laplacian.  For a more in depth discussion on the matrices $\Rf$ and $\Tcyclef$, please see \cite{Zelazo2009b, Zelazo2011}.  Note also that the matrix $L_e(\mc{F})^{-1}\Eforrest^T$ is the \emph{left-inverse} of $\Eforrest$; we denote this matrix as $\Eforrest^L$.
Proposition \ref{prop:edgelap_sim2} immediately leads to the first result on the signature of the weighted Laplacian and its relationship to the essential edge Laplacian.\footnote{We use a slight abuse of terminology by referring to the signature of $L_{\ms ess}(\mc{F})$ as it is not a symmetric matrix in general.  However, it is straight forward to show $L_{\ms ess}(\mc{F})$ is similar to a symmetric matrix, and thus the meaning of $\sigma(L_{\ms ess}(\mc{F}))$ is clear.}

\begin{theorem}\label{thm:signature1}
Assume $\mc{G}$ has $c$ connected components and $\sigma(L(\mc{G}))=(n_+,n_-,n_0)$.  Then $\sigma(L_{\ms ess}(\mc{F})) = (n_+,n_-,n_0-c)$.
\end{theorem} 
The result of Proposition \ref{prop:edgelap_sim2} and Theorem \ref{thm:signature1} shows that the presence of negative edge weights can introduce both negative and zero eigenvalues.  The next result relates the signature of the essential edge Laplacian matrix to the matrix $\Rf W\Rf^T$.

\begin{corollary}\label{cor:inertia}
$$ \sigma(L_{\ms ess}(\mc{F}))=\sigma(\Rf W \Rf^T).$$
 \end{corollary}

\begin{proof}
Using the similarity transformation matrix $L_e(\mc{F})^{\frac{1}{2}}$ we have that $\Edgelapf$ is similar to $L_e(\mc{F})^{\frac{1}{2}}\Rf W \Rf^TL_e(\mc{F})^{\frac{1}{2}}$.  This matrix is congruent to $\Rf W \Rf^T$ and thus has the same signature as $\Edgelapf$.   \QEDclosed
\end{proof}

\begin{corollary}\label{cor:inertia}
Assume $\mc{G}$ has $c$ connected components and $\sigma(L(\mc{G}))=(n_+,n_-,n_0)$.  Then $\sigma(\Rf W \Rf^T) = (n_+,n_-,n_0-c)$.
\end{corollary}

The matrix $\Rf W\Rf^T$ turns out to be closely related to many combinatorial properties of a graph.  For example, the rows of the matrix $\Rf$ form a basis for the cut-space of the graph \cite{Godsil2001}.  This matrix is also intimately related to the notion of effective resistance of a graph, which will be discussed in the sequel.  Corollary \ref{cor:inertia} thus shows that studying the definiteness of the weighted Laplacian can be reduced to studying the matrix $\Rf W \Rf^T$ which contains in a more explicit way information on how both the location and magnitude of negative weight edges influence it spectral properties.
\section{Effective Resistance and the\\ Definiteness of the Weighted Laplacian}\label{sec:definiteness}

The results of Section \ref{sec:signature} reveal that $\sigma(L(\mc{G}))$ is related to $\sigma(\Rf W \Rf^T)$.  In this section, we exploit the structure of this matrix to show how the negative edge weights affect the definiteness of the weighted Laplacian.  
The derived conditions turn out to be related to the notion of the \emph{effective resistance} of a graph.

It is well known that the weighted Laplacian of a graph can be interpreted as a resistor network \cite{Klein1993}.  
Each edge in the network can be thought of as a resistor with resistance equal to the inverse of the edge weight, $r_k = \mc{W}(k)^{-1}=w_k^{-1}$ for $k \in \mc{E}$.\footnote{Thus, the edge weight $w_k$ can be interpreted as an \emph{admittance}.}  The resistance between any two pairs of nodes can be determined using standard methods from electrical network theory \cite{Klein1993}.  It may also be computed using the Moore-Penrose pseudo-inverse of the graph Laplacian, denoted $L(\mc{G})^{\dagger}$.

\begin{definition}[\cite{Klein1993}]\label{def:effec_resistance}
The effective resistance between nodes $u,v \in \mc{V}$ in a weighted graph, denoted  $\mc{R}_{uv}(\mc{G})$, is 
\beas
\mc{R}_{uv}(\mc{G}) &=& ({\bf e}_u - {\bf e}_v)^TL^{\dagger}(\mc{G})({\bf e}_u - {\bf e}_v) \\
&=& [L^{\dagger}(\mc{G})]_{uu}-2[L^{\dagger}(\mc{G})]_{uv}+[L^{\dagger}(\mc{G})]_{vv},
\eeas
where ${\bf e}_u$ is the indicator vector for node $u$, that is ${\bf e}_u = 1$ in the $u$ position and 0 elsewhere.
\end{definition}

Our first result shows how the effective resistance between two nodes is related to the matrix $\R W \R$.  In this direction, we first derive an expression for the pseudo-inverse of the graph Laplacian using the essential edge Laplacian matrix.  
\begin{proposition}\label{prop:lap_pseudoinv}
Let $\mc{G}$ be a connected graph and assume $\sigma(L(\mc{G}))=(n_+,n_-,1)$.  Then the pseudo-inverse of the weighted graph Laplacian can be expressed as
\bea\label{lap_pseudoinv}
L^{\dagger}(\mc{G})&=& (\Etree^L)^T\left(\R W\R^T\right)^{-1}\Etree^L \nonumber \\
&=& (\Etree^L)^{T} L_{\ms ess}(\mc{T})^{-1}\Etree^T,
\eea
where $\Etree^L = L_e(\mc{T})^{-1}\Etree^T$ is the \emph{left-inverse} of $\Etree$.
\end{proposition}

\begin{proof}
From Theorem \ref{thm:signature1} we conclude the essential edge Laplacian is invertible and it follows that 
$$ L_{\ms ess}(\mc{T})^{-1} = \left(\R W\R^T\right)^{-1}L_e(\mc{T})^{-1},$$
and
$$L^{\dagger}(\mc{G}) = S \leftm{cc} \left(\R W\R^T\right)^{-1}L_e(\mc{T})^{-1} & {\bf 0} \\ {\bf 0} & 0 \rightm S^{-1},$$
where $S$ is the transformation matrix defined in Proposition \ref{prop:edgelap_sim2}, and (\ref{lap_pseudoinv}) follows directly. \QEDclosed
\end{proof}

From Proposition \ref{prop:lap_pseudoinv}, it is clear that the effective resistance between nodes $u,v \in \mc{V}$ can be expressed as
{\small
$$ \mc{R}_{uv}(\mc{G}) =\hspace{-2pt} ({\bf e}_u - {\bf e}_v)^T(\Etree^L)^T \left(\hspace{-2pt}\R W\R^T\hspace{-2pt}\right)^{-1}\hspace{-5pt}\Etree^L({\bf e}_u - {\bf e}_v).$$
}
We now show that this equivalent characterization of the effective resistance is useful for understanding the definiteness of the weighted Laplacian.

\begin{figure}[!t]
\begin{center}
\includegraphics[width=0.5\textwidth]{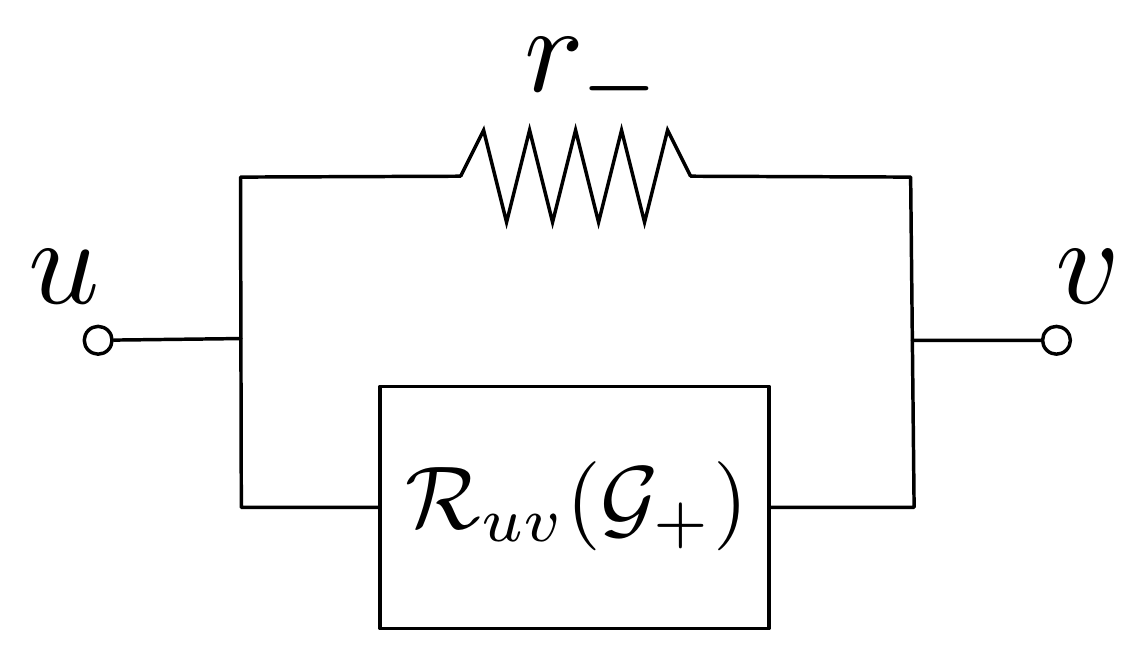}
\caption{Resistive network interpretation with one negative weight edge.}
\label{fig.effectiveresistance}
\end{center}
\end{figure}

\begin{theorem}\label{thm:one_negedge}
Assume that $\mc{G} = (\mc{V},\mc{E},\mc{W})$ has one edge with a negative weight, $e_-=(u,v) \in \mc{E}$.  Let $\mc{G}_+=(\mc{V},\mc{E}\setminus \{e_-\}, \mc{W})$ and $\mc{G}_- =(\mc{V},e_-,\mc{W})$  and assume $\mc{G}_+$ is connected.  Furthermore, let $\mc{R}_{uv}(\mc{G}_+)$ denote the effective resistance between nodes $u,v \in \mc{V}$ over the graph $\mc{G}_+$.  Then $L(\mc{G})$ is positive semi-definite if and only if $|\mc{W}(e_{-})| \leq \mc{R}_{uv}^{-1}(\mc{G}_+)$.
\end{theorem}
\begin{proof}
Denote by $E_-$ the incidence matrix of $\mc{G}_-$, and $E_+ = \Etreep\Rtplus $ the incidence matrix of $\mc{G}_+ = \mc{T}_+ \cup \mc{C}_+$.  The Laplacian matrix can now be expressed as 
$$L(\mc{G})= \Etreep\Rtplus W_+ \Rtplus^T\Etreep^T - E_-|\mc{W}(e_-)|E_-^T.$$
By the Schur complement, $L(\mc{G}) \geq 0$ if and only if
$$ \leftm{cc} |\mc{W}(e_-)|^{-1} & E_-^T \\ E_- & \Etreep\Rtplus W_+ \Rtplus^T\Etreep^T \rightm \geq 0.$$
Applying a congruent transformation to the above matrix using 
\beas
S &=& \leftm{cc} I & 0 \\ 0 & \leftm{cc} (\Etreep^L)^T &\ones \rightm \rightm
\eeas
leads to the following LMI condition,
$$ \leftm{cc} |\mc{W}(e_{-})|^{-1} &  E_{\_}^T(\Etreep^L)^T  \\   \Etreep^L E_{-} &  \Rtplus W_+ \Rtplus^T   \rightm  \geq 0.$$
Applying again the Schur complement, we obtain the equivalent condition that the matrix
{\small
$$ |\mc{W}(e_{-})|^{-1} - E_{-}^T(\Etreep^L)^T(\Rtplus W_+ \Rtplus^T)^{-1}\Etreep^L E_{\_}$$
}
must also be positive semi-definite.
Observe now that 
{\small
$$E_{-}^T(\Etreep^L)^T(\Rtplus W_+ \Rtplus^T)^{-1}\Etreep^L E_{-} = \mc{R}_{uv}(\mc{G}_+).$$
}
This then leads to the desired conclusion that $ |\mc{W}(e_{-})|  \leq  \mc{R}_{uv}^{-1}(\mc{G}_+)$.\QEDclosed
\end{proof}

The above result has a very intuitive physical interpretation.  The entire network $\mc{G}_+$ can be considered as a single lumped resistor between nodes $u$ and $v$ with resistance $\mc{R}_{uv}(\mc{G}_+)$.  The negate-weight edge can thus be thought of adding another resistor in parallel between the nodes, as in Figure \ref{fig.effectiveresistance}.  The equivalent resistance between $u$ and $v$ is well-known to be 
$$\mc{R}_{uv}(\mc{G}) = \frac{\mc{R}_{uv}(\mc{G}_+)r_{-}}{\mc{R}_{uv}(\mc{G}_+)+r_{-}}.$$
If $r_{-}$ is a negative resistor, then choosing $r_{-} = -\mc{R}_{uv}(\mc{G}_+)$ corresponds to an equivalent resistance that is infinite, i.e., an \emph{open circuit}.  The open circuit can be thought of as a cut between the terminals $u$ and $v$.  

The result in Theorem \ref{thm:one_negedge} can be generalized to multiple negative weight edges with some additional assumptions on how those edges are distributed in the graph. In this direction, let $\mc{E}_-$ and $\mc{E}_+$ denote, respectively, the edges with negative and positive weights.
For each edge $k=(u,v)\in \mc{E}_-$, define the set $P_{k} \subseteq \mc{E}_+$ to be the set of all edges in $\mc{G}_+=(\mc{V},\mc{E}_+)$ that belong to a path connecting nodes $u$ to $v$, 
\bea\label{disjoint_paths}
P_{k} &=& \left\{ e \in \mc{E}_+ \, | \, k=(u,v)\in \mc{E}_-, \exists \mbox{ a path in } \mc{G}_+ \right. \nonumber \\ 
&&\left. \mbox{ from } u \mbox{ to } v \mbox{ using edge } e \right\}.
\eea
Let $\mc{G}_+(P_{k}) \subseteq \mc{G}_+$ be the subgraph induced by the edges in $P_{k}$.\footnote{Thus, $\mc{G}_+(P_{k}) = (\mc{V}(P_{k}),P_{k})$ where $\mc{V}(P_{k})\subseteq \mc{V}$ are the nodes incident to edges in $P_{k}$.}
Note that if $P_{k} \cap P_{{k'}} = \emptyset$ for edges with distinct nodes (i.e. $k=(u,v)$ and ${k'}=({u'},{v'}) \in \mc{E}_-$), then there exists no cycle in $\mc{G}_+$ containing the nodes $u,v,{u'},{v'}$.
An important class of graphs that can admit such a partition are the \emph{cactus graphs} \cite{Markov2007}.  Using this characterization, the following statement on effective resistance with multiple negative weight edges can be stated as follows.

\begin{theorem}\label{thm:multiple_edge}
Assume that $\mc{G}_+$ is connected and $|\mc{E}_-|>1$.  Let $\mc{R}_{k}(\mc{G}_+)$ denote the effective resistance between nodes $u,v \in \mc{V}$ with $k=(u,v)\in \mc{E}_-$ over the graph $\mc{G}_+$, and let $\mc{\bf R} = \mbox{\bf diag}\{\mc{R}_1(\mc{G}_+),\ldots,\mc{R}_{|\mc{E}_-|}(\mc{G}_+)\}$.  Furthermore, assume that $P_i \cap P_j = \emptyset$ for all $i,j \in \mc{E}_-$, where $P_i$ is defined in (\ref{disjoint_paths}).  Then the weighted Laplacian is positive semi-definite if and only if $|W_-| \leq \mc{\bf R}^{-1}$.
\end{theorem}
\begin{proof}
As in the proof of Theorem \ref{thm:one_negedge}, we consider the LMI
{\small
$$|W_-|^{-1} - E_{-}^T(\Etreep^L)^T(\Rtplus W_+ \Rtplus^T)^{-1}\Etreep^L E_{-} \geq 0 $$
}
Due to the location of the negative weight edges assumed in the graph, it can be verified that the matrix $E_{-}^T(\Etreep^L)^T(\Rtplus W_+ \Rtplus^T)^{-1}\Etreep^L E_{-}$ is in fact a diagonal matrix with $\mc{R}_{k}(\mc{G}_+)$ for $k=1,\ldots, |\mc{E}_-|$ on the diagonal, denoted as ${\bf R}$.  To see this, observe that $\Etreep^LE_{-}$ is a $\{0,\pm 1\}$ matrix that describes which edges in the spanning tree $\mc{T}_+$ can create a cycle with each edge in $\mc{E}_{-}$ (this is related to the matrix $\Tcycle$ used in Proposition \ref{prop:edgelap_sim2} since $\mc{T}_+=\mc{T}$ and therefore $\mc{E}_- \subseteq \mc{E}_c$).  Observe also that an edge $k \in \mc{E}_-$ can only be incident to nodes in the subgraph $\mc{G}_+(P_k)$.  Therefore, the matrix $\Etreep^LE_{-}$ has a partitioned structure (after a suitable relabeling of the edges) such that the $k$th column of $\Etreep^LE_{-}$ will contain non-zero elements corresponding to edges in $\mc{T}_+ \cap \mc{G}_+(P_k)$.

The LMI condition can now be expressed as $|W_-|^{-1} \geq {\bf R}$ which implies that $|W_-| \leq  {\bf R}^{-1}$ concluding the proof.\QEDclosed
\end{proof}

Theorem \ref{thm:multiple_edge} also has the same physical interpretation as Theorem \ref{thm:one_negedge}.  Indeed, the resistance between two nodes contained in a sub-graph $\mc{G}_+(P_k)$ is not determined by any other edges in the network.  Both Theorems \ref{thm:one_negedge} and \ref{thm:multiple_edge} provide a clear characterization of how negative weight edges can impact the definiteness of the weighted Laplacian, and how that is related to the effective resistance in the graph.  In fact, we also can observe an additional property relating the total effective resistance between all nodes incident to edges in $\mc{E}_-$ and the definiteness of the graph, independent of the actual location of these edges in the network.

\begin{corollary}\label{cor:total_resist}
Assume that $\mc{G}_+$ is connected.  If the weighted Laplacian is positive semi-definite, then  
$$ \sum_{k \in \mc{E}_-} |\mc{W}(k)|^{-1} \geq {\bf R}_{tot},$$
where 
{\small 
$${\bf R}_{tot}\hspace{-3pt}=\hspace{-3pt}\mbox{\bf trace}\left[ E_{-}^T(\Etreep^L)^T(\Rtplus W_+ \Rtplus^T)^{-1}\Etreep^L E_{-}\right] .$$
}
\end{corollary}

Corollary \ref{cor:total_resist} indicates that a weighted Laplacian with negative weights can still be positive semi-definite, and in that case the total magnitude of the negative weight edges is closely related to the total effective resistance in the network (defined over the nodes incident to $\mc{E}_-$). The notion of total effective resistance has also appeared in works characterizing the $\mc{H}_2$ performance of certain multi-agent networks \cite{Bamieh2009, Barooah2006a, Siami2013}.  While Corollary \ref{cor:total_resist} only provides a sufficient condition for the definiteness of the weighted Laplacian, it nevertheless reinforces its connection to the notion of effective resistance.

\section{Clustering with Negative Weights}\label{sec:example}

\begin{figure}[!t]
\begin{center}
	\subfigure[]{\scalebox{.6}{\includegraphics{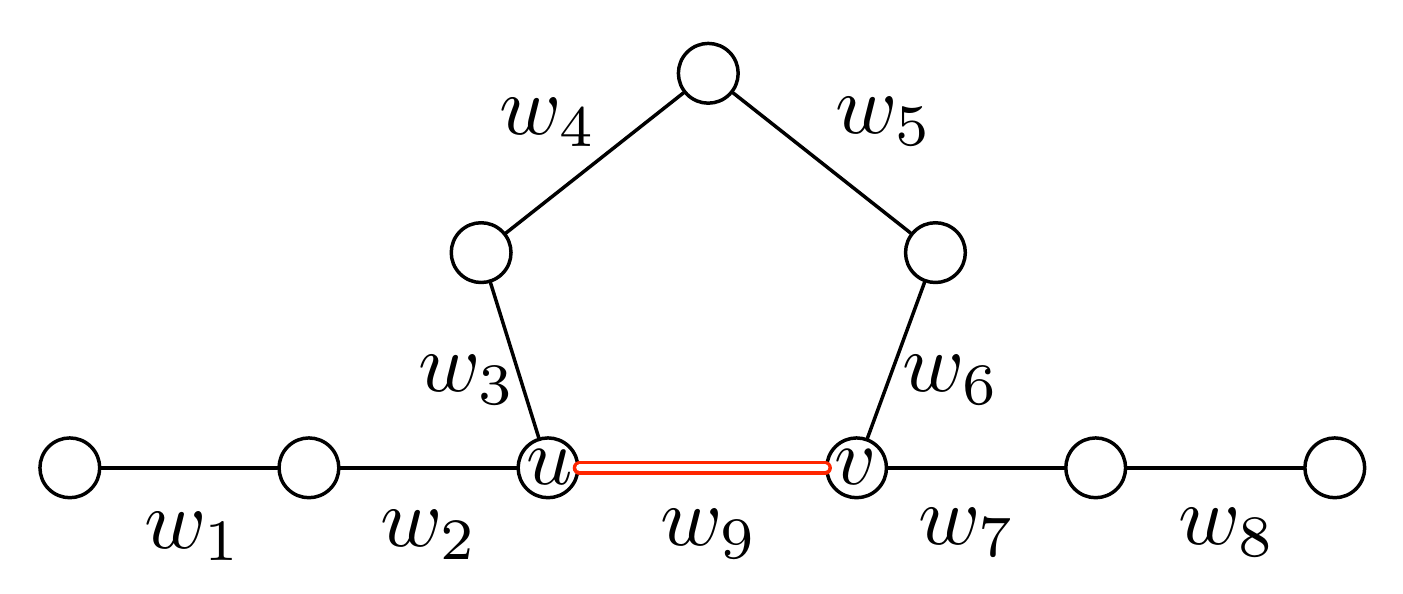}}\label{fig:cycle}}
\caption{A graph without and with a negative weight edge (in red/double line).}
\label{fig.graph1cycle}
\end{center}
\end{figure} 

In this section we demonstrate how the previous results can be used to design edge weights for a linear weighted consensus protocol that results in bounded trajectories that are \emph{clustering}.  That is, the agents comprising the system form clusters, and agents within a single cluster reach agreement on a common state that is different than agents in other clusters.

In this direction, we consider the linear weighted consensus protocol over a weighted and undirected graph on $|\mc{V}|=n$ nodes, $\mc{G}=(\mc{V},\mc{E},\mc{W})$ \cite{Mesbahi2010},
\bea\label{consensus}
\dot{x}(t) = -L(\mc{G})x(t).
\eea

As an illustrative example, consider the graph in Figure \ref{fig.graph1cycle} (without the edge $(u,v)$) with edge weights $w_i=1, i=1,\ldots,8$.  It can be verified that the effective resistance between nodes $u$ and $v$ is $\mc{R}_{uv}(\mc{G}) = 4$.  Consider now the graph $\mc{\hat{G}}=(\mc{V},\mc{E} \cup \{(u,v)\},\hat{\mc{W}})$, and assume that the added edge has a negative weight (i.e., $w_{uv}=\hat{\mc{W}}((u,v))<0$). Theorem \ref{thm:one_negedge} can now be used to conclude that any edge weight satisfying $w_{uv} \geq -0.25 $ guarantees that the weighted Laplacian will be positive semi-definite.  

In the context of the weighted consensus protocol, this result can be used to produce very different trajectories of the system.  For example, even in the presence of a negative weight, the agreement protocol over the graph $\hat{\mc{G}}$ can still reach agreement among all agents.  Figure \ref{fig:sync_negweight} demonstrates this using $w_{uv} = -0.1$ as the weight, and $\sigma(L(\hat{\mc{G}}))=(8,0,1)$.  More interesting are the trajectories generated by the consensus protocol when the negative weight edge is exactly matched to the effective resistance between the incident nodes.  Figure \ref{fig:cluster_negweight} shows the trajectories of the system for edge weight $w_{uv}=-0.25$.  In this case it can be verified that $L(\hat{\mc{G}})$ is still positive semi-definte, but the multiplicity of the zero-eigenvalue has increased, i.e., $\sigma(L(\hat{\mc{G}}))=(7,0,2)$.  The trajectories generate a clear clustered structure.

\begin{figure}[!t]
\begin{center}
	\subfigure[Synchronization is achieved even with a negative edge weight ($w_{uv}=-0.1$).] {\scalebox{.4}{\includegraphics{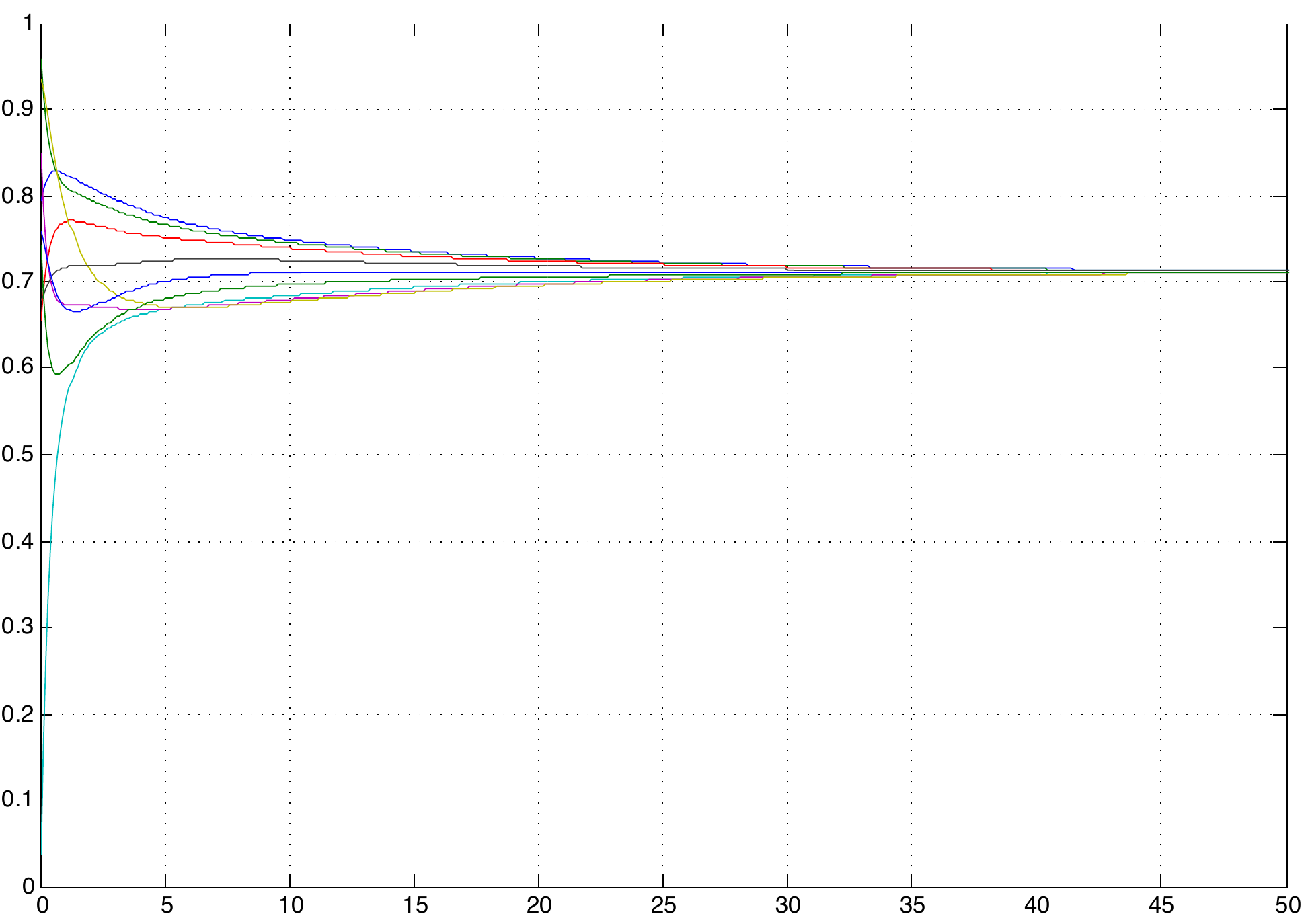}}\label{fig:sync_negweight}}
	\subfigure[Cluster synchronization from a negative edge weight ($w_{uv}=-0.25$).]{\scalebox{.4}{\includegraphics{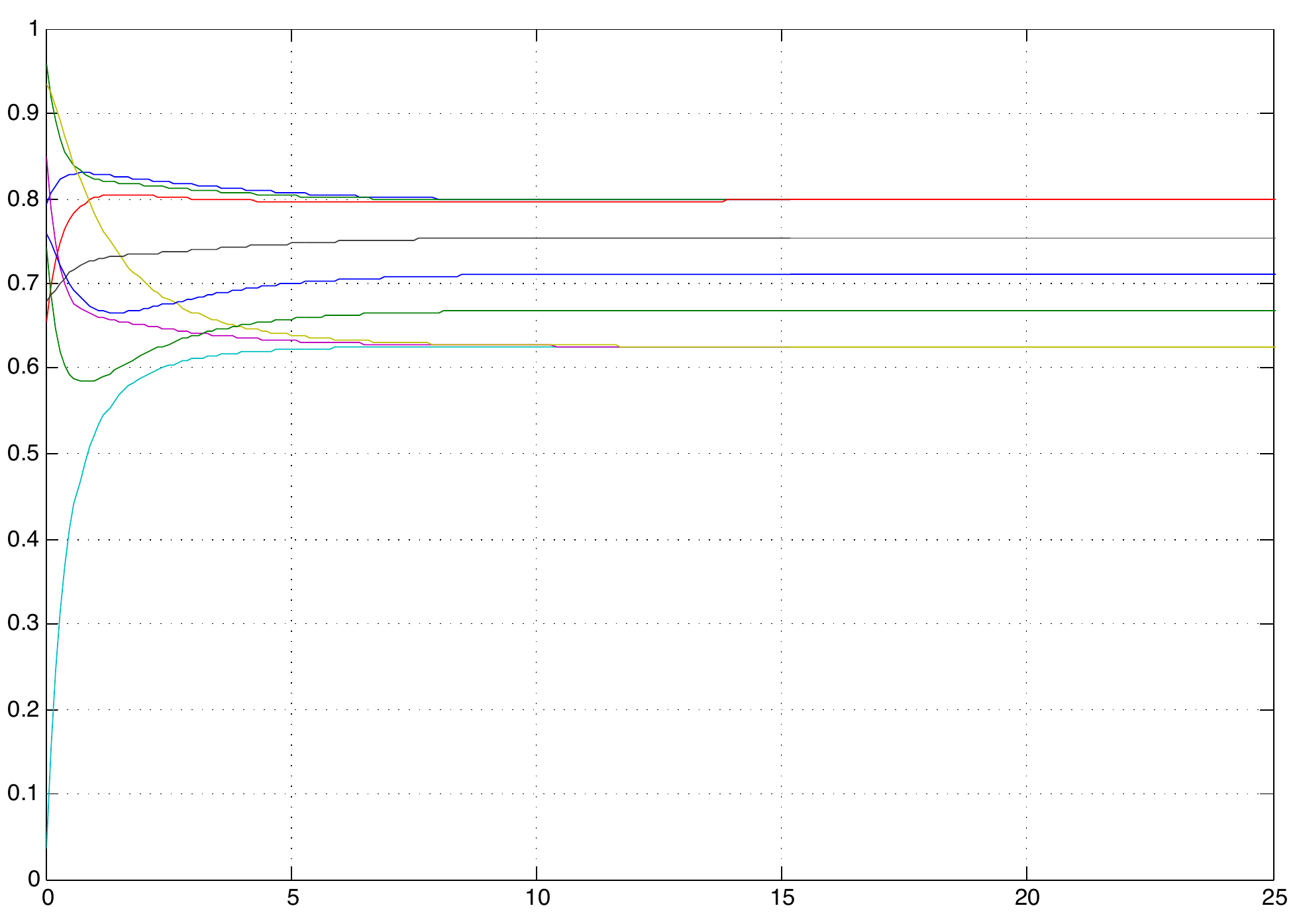}}\label{fig:cluster_negweight}}
  \caption{The consensus protocol for the graph in Figure \ref{fig.graph1cycle}.}\label{fig:negweights_consensus}
\end{center}
\end{figure}

In fact, using the results form Theorem \ref{thm:multiple_edge}, we can formulate a very precise statement regarding the clustering structure of a weighted consensus protocol with negative edge weights.  Due to space limitations, we provide here a proof for the clustering structure obtained by a graph with a single cycle and single negative weight edge.

\begin{proposition}\label{prop:cluster1cycle}
Consider the consensus protocol (\ref{consensus}) over a connected weighted graph $\mc{G}$.  Assume that $\mc{G}$ contains only one cycle, $\mc{G}_+$ is connected, and there is only one negative weight edge $e_- = (u,v)$ with $|\mc{W}_-(e_-)| = \mc{R}_{uv}^{-1}(\mc{G}_+)$.  Then for any initial condition, the trajectories generated by (\ref{consensus}) form $q$ clusters, where $q$ is the number of components in the graph obtained by removing all the edges in $\mc{G}$ contained in the cycle.
\end{proposition}
\begin{proof}
Observe that $\mc{G}_+$ is in fact a spanning tree ($\mc{T}$), and the cycle subgraph $\mc{C}=\mc{G}_-$.  By Corollary \ref{cor:inertia} it follows that the matrix $\R W \R$ has one less eigenvalue at the origin than $L(\mc{G})$.  We now show that $\R W \R$ has one eigenvalue at the origin.  From the assumption on the structure of $\mc{G}$, it can be shown that $[\Tcycle]_i=1$ whenever edge $i$ in $\mc{T}$ is in the cycle, and 0 otherwise.  Thus, $\R W \R = W_+ + \mc{W}_-(e_-) \Tcycle \Tcycle^T$ is similar to $\mc{R}_{uv}(\mc{G}_+)I-W_{+}^{-\frac{1}{2}}\Tcycle \Tcycle^TW_{+}^{-\frac{1}{2}}$.  The matrix $W_{+}^{-\frac{1}{2}}\Tcycle \Tcycle^TW_{+}^{-\frac{1}{2}}$ is a rank-one matrix with eigenvalue equal precisely to $\mc{R}_{uv}(\mc{G}_+)$ showing that $\R W \R $ has only one eigenvalue at the origin.

Having verified that $L(\mc{G})$ contains two eigenvalues at the origin, we are now able to explicitly construct a null-space eigenvector orthogonal to $\ones$.  Such a vector must satisfy 
$$E^Tx = W^{-1}\leftm{c} \Tcycle \\ -1\rightm.$$  
This vector will have a characteristic structure such that all entries corresponding to nodes in the cycle have unique values and sum to zero, and the remaining entries must be constant corresponding to each component obtained by removing all the edges in $\mc{G}$ contained in the cycle.  The trajectories generated by (\ref{consensus}) will thus reach agreement on each of these components resulting in the claimed clustering structure.\QEDclosed
\end{proof}

\section{Concluding Remarks}\label{sec:conclusion}

This work provided an analysis of the definiteness of the weighted Laplacian with negative edge weights.  It was shown that the signature of the weighted Laplacian is related to two special matrices, the essential edge Laplacian and the matrix $\Rf W \Rf^T$.  These matrices were then shown to be intimately related to the notion of effective resistance in a graph.  In this way, we could conclude that the definiteness of the weighted Laplacian depends on both the magnitude of the negative edge weights and their location in the graph.  In particular, a single negative edge weight must have a magnitude inversely proportional to the effective resistance between the incident nodes to produce an indefinite weighted Laplacian.  These results were also extended to more general scenarios, and their utility demonstrated on a linear weighted consensus protocol.

{\small
\section*{Acknowledgements}
The work presented here has been supported by the Arlene \& Arnold Goldstein Center at the TechnionÕs Autonomous System Program (TASP) and the Israel Science Foundation.

{\footnotesize
 \bibliographystyle{IEEEtran}
\bibliography{LaplacianMatrix}
}
}
\end{document}